\documentclass[12pt,reqno]{article}
\usepackage[letterpaper,margin=1.6cm]{geometry}
\usepackage{latexsym, amsfonts, amsthm,amssymb,amscd,amsmath,makeidx}
\usepackage{enumerate}
\usepackage[normalem]{ulem}
\usepackage{exscale}
\usepackage{overpic} 
\usepackage{color} 
\usepackage{graphicx}
\usepackage{overpic}  
\usepackage{bm}
\usepackage{comment}
\usepackage{caption}
\usepackage{float} 
\usepackage{array} 
\usepackage{booktabs} 
\usepackage{tabularx}

\usepackage{mathrsfs}
\usepackage{amssymb}
\usepackage{mathtools}

\usepackage{tikz}
\usepackage{blindtext}
\usepackage{bbm}
\usepackage{fancyhdr} 
\usepackage{mathrsfs}
\usepackage{wrapfig}
\usepackage{hyperref}
\usepackage{bookmark}
\usepackage{cleveref}
\usepackage{cases}
\usepackage{multirow}
\usepackage{pdflscape}
\usepackage[title]{appendix}

\hypersetup{
	colorlinks,
	citecolor=red,
	filecolor=black,
	linkcolor=blue,
	urlcolor=black
}

\definecolor{DarkGreen}{rgb}{0.2,0.6,0.2}

\def\Om{\Omega}\def\om{\omega}

\newcolumntype{Y}{>{\centering\arraybackslash}X}

\def\wt{\widetilde}

\newcolumntype{C}{>{\centering\arraybackslash}X}

\def\ignore#1{}

\def\bR{{\mathbb R}}

\def\bP{{\mathbb P}}
\def\bE{{\mathbb E}}

\numberwithin{equation}{section}

\usetikzlibrary{positioning,calc,cd}

\allowdisplaybreaks[4]

\def\cD{{\mathscr D}}
\def\cF{{\mathscr F}}

\newtheorem{theorem}{Theorem}[section]
\newtheorem{proposition}[theorem]{Proposition}
\newtheorem{lemma}[theorem]{Lemma}

\theoremstyle{definition}
\newtheorem{definition}[theorem]{Definition}
\newtheorem{example}[theorem]{Example}

\newtheorem{remark}[theorem]{Remark}

\def\<{\langle}
\def\>{\rangle}
\def\wt#1{\widetilde{#1}}

\def\wt{\widetilde}

\begin{document}

	\title{A criterion for absolute continuity relative to the law of fractional Brownian motion}
	\author{Xiyue Han and Alexander Schied\thanks{
			Department of Statistics and Actuarial Science, University of Waterloo, 200 University Ave W, Waterloo, Ontario, N2L 3G1, Canada. E-Mails: {\tt xiyue.han@uwaterloo.ca, aschied@uwaterloo.ca}.\hfill\break
			The authors gratefully acknowledge support from the
			Natural Sciences and Engineering Research Council of Canada through grants RGPIN-2017-04054 and RGPIN-2024-03761.}}
	\date{\normalsize First version: June 5, 2023\\
	This version: November 16, 2024}
	\maketitle

\begin{abstract}Let $X$ be the sum of a fractional Brownian motion with Hurst parameter $H$ and an absolutely continuous and adapted drift process. 
We establish a  simple criterion that guarantees that the law of $X$ is absolutely continuous with respect to the law of the original fractional Brownian motion. For $H<1/2$, the trajectories of the derivative of the drift need to be bounded by an almost surely finite random variable; for $H>1/2$, they need to satisfy a H\"older condition with some exponent larger than $2H-1$. These are almost-sure conditions,  and no expectation requirements are imposed. For the case in which $X$ arises as the solution of a nonlinear stochastic integral equation driven by fractional Brownian motion, we provide a simple criterion on the drift coefficient under which the law of $X$ is automatically equivalent to the one of fractional Brownian motion. 
\end{abstract}

\medskip


\section{Introduction and statement of the main result}

In recent decades, fractional Brownian motion has emerged as a fundamental mathematical tool  for modeling scale-dependent random processes, exhibiting long-range dependence and self-similarity. Originating from the seminal work of Beno\^it Mandelbrot in the early 1960s,  fractional Brownian motion has since captured the attention of researchers across numerous scientific disciplines due to its versatile applications in fields such as physics, finance, telecommunications, and image processing; see the monographs \cite{AscioneMishuraPirozzi,KubiliusMishura,Mishura,Nourdin2012,UmarovEtAl} and the references therein for a comprehensive account.

Fractional Brownian motion with Hurst parameter $H\in(0,1)$ is a centered Gaussian stochastic process $W^H$ with continuous sample paths and  covariance function
$$\bE[W^H_sW^H_t]=\frac12\big(|s|^{2H}+|t|^{2H}-|t-s|^{2H}\big), \qquad s,t\ge0,
$$
on a given probability space $(\Om,\cF,\bP)$. For $H=1/2$, we recover standard Brownian motion  $W=W^{1/2}$ as a special case. 

A classical research question concerns the characterization of stochastic processes $X$ whose law is absolutely continuous with respect to the one of $W^H$. Its solutions have numerous applications in the statistics of stochastic processes, the analysis of stochastic differential equations, and mathematical finance. A large portion of the two-volume monograph \cite{LiptserShiryaevI,LiptserShiryaevII} by Liptser and Shiryaev is devoted to the study of this question for standard Brownian motion $W$. It applies to the setting in which $(\Om,\cF,(\cF_t),\bP)$ is a probability space satisfying the usual conditions, $W$ is a Brownian motion relative to the filtration $(\cF_t)$, and $X$ is an It\^o process of the form
\begin{equation}\label{X Ito eq}
X_t=W_t+\int_0^t\xi_s\,ds,\qquad 0\le t\le T,
\end{equation}
for some progressively measurable  process $\xi$ with $\int_0^T\xi_t^2\,dt<\infty$ $\bP$-a.s.~and some $T>0$.  Then the traditional approach to absolute continuity uses Girsanov's theorem to obtain a candidate for the density $\varphi$ of the law of $X$ with respect to the Wiener measure in the following exponential form,
\begin{equation}
\varphi=\exp\bigg(\int_0^T\xi_t\,dW_t-\frac12\int_0^T\xi_t^2\,dt\bigg).
\end{equation}
The main difficulty with this approach is that it requires the verification of the condition $\bE[\varphi]=1$, which can be quite hard in concrete applications. As long as one only needs absolute continuity of the law and not the specific form of the density, one can rely on the following remarkable result, which was first obtained by Kadota and Shepp \cite{KadotaShepp} and later improved in \cite{Yershov,LipcerSirjaev}; it can also be found  in \cite[Theorem 7.2]{LiptserShiryaevI}. It is remarkable insofar as no integrability conditions with respect to the measure $\bP$ are imposed on the drift $\xi$.

\begin{theorem}[Kadota and Shepp  \cite{KadotaShepp}]\label{KadotaShepp thm} If $X$ is as in \eqref{X Ito eq} and $\int_0^T\xi_t^2\,dt<\infty$ $\bP$-a.s., then the law of $(X_t)_{0\le t\le T}$ is absolutely continuous with respect to the Wiener measure on $C[0,T]$.
\end{theorem}

Even more can be said in case $X$ is a process satisfying the stochastic integral equation
\begin{equation}\label{SDE}
X_t=W_t+\int_0^t\gamma(s,X)\,ds,\qquad 0\le t\le T,
\end{equation}
where $\gamma$ is progressively measurable with respect to the natural filtration of $W$. See \cite[Theorem 3]{Yershov} or \cite[Theorem 7.7]{LiptserShiryaevI} for a proof of the following result.

\begin{theorem}[Yershov \cite{Yershov}]\label{Yershov thm} If $X$  as in \eqref{SDE} is adapted to $(\cF_t)$, and $\int_0^T\gamma(t,X)^2\,dt$ and $\int_0^T\gamma(t,W)^2\,dt$ are both $\bP$-a.s.~finite, then the law of $(X_t)_{0\le t\le T}$  is equivalent to the one of $(W_t)_{0\le t\le T}$ 
\end{theorem}

In the present paper, our goal is to obtain sufficient conditions for the absolute continuity of the stochastic process
\begin{equation}\label{XH eq}
X_t:=x_0+W^H_t+\int_0^t\xi_s\,ds,\qquad0\le t\le T,\end{equation}
for the case in which the Hurst parameter $H$ is not equal to $1/2$ and $x_0\in\bR$. Here, we assume that $\xi$ is progressively measurable with respect to the filtration $(\cF_t)$.  If, in addition, the drift $\xi$ is of the form $\xi_t=b(X_t)$ for some function $b:\bR\to\bR$, our aim is to establish sufficient conditions for the equivalence of the law of $X$ to the one of $x_0+W^H$.
An important special case is the fractional Ornstein--Uhlenbeck process, which can be defined as the pathwise solution of the following integral equation 
\begin{equation}\label{OU process eq}
X_t=x_0+\rho\int_0^t(m-X_s)\,ds+W^H_t,\qquad t\ge0,
\end{equation}
where $\rho$ and $m$ are real parameters; see \cite{CheriditoFOU}. 

There exists also a Girsanov-type theorem for fractional Brownian motion. It was developed by Norros, Valkeila, and Virtamo \cite{Norrosetal1999} and works by transforming $W^H$ into a certain standard Brownian motion, to which one can then apply the traditional Girsanov theorem. However, in this context, verifying the condition $\bE[\varphi]=1$ for the corresponding stochastic exponential becomes even more involved, because  the drift process also needs to be transformed in an intricate manner. For this reason, we believe that easily verifiable criteria for absolute continuity may come in handy in many applications. 

Before stating our main results, however, we need to specify our requirements on the filtration $(\cF_t)$, with respect to which both $W^H$ and $\xi$ are adapted. Recall that in the Brownian case, $H=1/2$, the condition required for \Cref{KadotaShepp thm} is that $W$ is a Brownian motion relative to $(\cF_t)$. That is, the increment  $W_t-W_s$ is independent of $\cF_s$ for $0\le s<t$. Since the increments of fractional Brownian motion $W^H$ for $H\neq1/2$ are not independent, the preceding condition cannot be translated verbatim. Instead, we propose the following \Cref{Wh rel Ft def}, based on the process
\begin{equation}\label{wt M eq}
\wt M_t:=\int_0^ts^{1/2-H}(t-s)^{1/2-H}\,dW^H_s,\qquad 0\le t\le T.
\end{equation}
It was shown in \cite{Norrosetal1999} that the stochastic integral on the right exists $\bP$-a.s.~in a pathwise sense and that $\wt M$ is a Gaussian martingale and hence has independent increments. Clearly, $\wt M=W^H=W$ for $H=1/2$. 

\begin{definition}\label{Wh rel Ft def} We say that $W^H$ is a \emph{fractional Brownian motion relative to the filtration $(\cF_t)$} if $W^H$ is adapted to $(\cF_t)$ and the increment  $\wt M_t-\wt M_s$ is independent of $\cF_s$ for $0\le s<t$.
\end{definition}

Throughout the remainder of the paper, we will always assume that $W^H$ is a fractional Brownian motion relative to the given filtration $(\cF_t)$. Since 
it follows from \cite[Proposition 3.1]{Norrosetal1999} that   $W^H$ is a fractional Brownian motion with respect to its natural filtration, the assumption includes the case in which $(\cF_t)$ is the natural filtration of $W^H$. 
Now we can state our main result.

\begin{theorem}\label{fBM KadotaShepp thm} Let $X$ be given by \eqref{XH eq} and suppose  the following assumption is satisfied.
\begin{itemize}
\item If $H<1/2$, we assume  that the function $t\mapsto\xi_t$ is $\bP$-a.s.~bounded in the sense that  there exists a finite random variable $C$ such that $|\xi_t(\om)|\le C(\om)$ for a.e.~$t$ and $\bP$-a.e.~$\om$. 
\item If $H>1/2$, we assume that for $\bP$-a.s.~$\om$ the function $t\mapsto \xi_t(\om)$ is H\"older continuous with some exponent $\alpha(\om)>2H-1$.
\end{itemize}
 Then the law of $(X_t)_{t\in[0,T]}$ is absolutely continuous with respect to the law of $(x_0+W^H_t)_{t\in[0,T]}$.
\end{theorem}

 We emphasize that  in \Cref{fBM KadotaShepp thm} no integrability conditions are imposed on $\xi$ or the random variable $C$. Nevertheless, as illustrated by \Cref{KadotaShepp thm}  for $H=1/2$, the conditions in \Cref{fBM KadotaShepp thm}  may not be  optimal and can probably be improved. However, our goal was not to state and prove the most general theorem on absolute continuity, but to provide conditions for absolute continuity that can be verified with minimal effort and that apply in the most important applications. We believe this is the case for the conditions presented  here, and we illustrate this point by means of the following example and the subsequent result.

\begin{example}\label{Hurst ex}
The motivation for this research originated in the context of  \cite{HanSchiedHurstDerivative}, where we prove the strong consistency of a statistical  estimator  that estimates the parameter $H$ from discrete observations of the process
$$Y^{W^H}_t:=\int_0^tg(W^H_s)\,ds,\qquad 0\le t\le T,
$$
where $g\in C^2(\bR)$ is strictly monotone. This problem  arises from the desire to estimate the \lq roughness' of a rough stochastic volatility model (see, e.g.,  \cite{GatheralRosenbaum}) from its realized quadratic variation. Rough stochastic volatility models  often involve an additional drift, i.e., they are based on $X$ as in \eqref{XH eq} or \eqref{OU process eq} instead of $W^H$ and so one will observe the process
$$Y^X_t:=\int_0^tg(X_s)\,ds,\qquad 0\le t\le T. 
$$
If we know that the law of $X$ is absolutely continuous with respect to the law of $W^H$ for every given $H\in(0,1)$, then we can conclude immediately that the strong consistency of the estimator carries over from the process $Y^{W^H}$ to the process $Y^X$. Clearly, analogous statements hold for other strongly consistent statistical estimators in the context of fractional Brownian motion; see, e.g.,  \cite{KubiliusMishura} for examples. 
\end{example}
 
 The next theorem  applies in particular to the fractional Ornstein--Uhlenbeck process \eqref{OU process eq}.
 It is basically a corollary to \Cref{Yershov thm} and our proof of \Cref{fBM KadotaShepp thm}.
 
 \begin{theorem}\label{fBM Yershov theorem}Suppose that $X$ solves the following integral equation,
 \begin{equation}\label{fSDE eq}
 X_t=x_0+W^H_t+\int_0^tb(X_s)\,ds,\qquad 0\le t\le T,
 \end{equation}
 for some locally bounded and measurable function  $b:\bR\to\bR$ and that $X$ is adapted to $(\cF_t)$. For $H>1/2$ we assume in addition that $b$ is locally H\"older continuous with some exponent $\alpha>2-1/H$.  Then the law of $(X_t)_{t\in[0,T]}$ is equivalent to the law of $(x_0+W^H_t)_{t\in[0,T]}$.
 \end{theorem}
 
 A criterion for the existence of the weak solution to the nonlinear integral equation \eqref{fSDE eq} in the case $H\le1/2$ is given in \cite{Anzeletti}.

\begin{remark}\label{density remark 1}
  As a matter of fact, in the context of \Cref{fBM Yershov theorem} it is even possible to provide a formula for the density between the two laws. This formula, however, 
involves a number of quantities that will only be introduced in the proof of \Cref{fBM KadotaShepp thm}. For this reason, it is deferred to \Cref{density remark 2}
 \end{remark}
 
 \begin{example}\label{arbitrage ex} Cheridito \cite{CheriditoArbitrage} showed that a financial market model based on fractional Brownian motion with deterministic drift does not admit arbitrage if trading is restricted to a certain class of strategies with finitely many rebalancing times in $[0,T]$. Since the absence of arbitrage is preserved under an equivalent measure change, \Cref{fBM Yershov theorem} shows that this no-arbitrage result carries over to market models built on a process of the form \eqref{fSDE eq}.
 \end{example}

\section{Proofs}
\subsection{Proof of \Cref{fBM KadotaShepp thm}}

In this  section, we prove \Cref{fBM KadotaShepp thm}. We first note that we may clearly assume without loss of generality that $x_0=0$. Our approach to proving \Cref{fBM KadotaShepp thm} combines \Cref{KadotaShepp thm} with 
the transformation theory developed by Norros et al.~\cite{Norrosetal1999}.
  Drawing on the latter reference, we define the following constants,
$$c_H=\sqrt{\frac{2H\Gamma(3/2-H)}{\Gamma(H+1/2)\Gamma(2-2H)}},\quad c_1=\frac1{2H\Gamma(\frac32-H)\Gamma(H+\frac12)},\quad c_2=\frac{c_H}{2H(2-2H)^{1/2}}.
$$
We also let
\begin{align*}
w(t,s)&=c_1s^{1/2-H}(t-s)^{1/2-H},\qquad 0\le s\le t\le T.
\end{align*}
It is shown in  \cite{Norrosetal1999} that the following integrals exist $\bP$-a.s.~in a pathwise sense,
\begin{equation}\label{M and Y eq}
Y_t=\int_0^ts^{1/2-H}\,dW^H_s\quad\text{and}\quad  M_t=\int_0^tw(t,s)\,dW^H_s.
\end{equation}
Note that $M=c_1\wt M$, where $\wt M$ is as in \eqref{wt M eq}.
 Theorem 3.1 in \cite{Norrosetal1999}  states that $M$ is a  continuous Gaussian martingale with quadratic variation $\<M\>_t=c_2^2t^{2-2H}$. Since
$$\int_0^t(s^{H-1/2})^2\,d\<M\>_s=c_2^2(2-2H)t=\frac{c_H^2t}{4H^2},
$$
 the process
\begin{equation}\label{BM B def}
B_t:=\frac{2H}{c_H}\int_0^ts^{H-1/2}\,dM_s
\end{equation}
is a standard Brownian motion. Theorem 5.2 in \cite{Norrosetal1999} states that $W^H$ can be recovered from $B$ as follows,
\begin{equation}\label{W from B eq}
W^H_t=\int_0^t\zeta(t,s)\,dB_s,
\end{equation}
where for $0\le s<t\le T$,
\begin{equation*}
\zeta(t,s)=c_H\bigg[\Big(\frac ts\Big)^{H-1/2}(t-s)^{H-1/2}-(H-1/2)s^{1/2-H}\int_s^tu^{H-3/2}(u-s)^{H-1/2}\,du\bigg] .
\end{equation*}
It follows that there is a one-to-one correspondence between $W^H$ and $B$. 
The proof of the identity \eqref{W from B eq} is based on the following two relations, which are derived   in Theorem 3.2 and Remark 3.1 of  \cite{Norrosetal1999}, \begin{equation}\label{M and Y relations eq}
M_t=c_1\int_0^t(t-s)^{1/2-H}\,dY_s\quad\text{and}\quad  Y_t=2H\int_0^t(t-s)^{H-1/2}\,dM_s.
\end{equation}

Our goal is to extend these arguments to the process $X$. We start with the following simple lemma, which extends  \cite[Theorem 4.1]{Norrosetal1999} to the case of a constant, $\cF_0$-measurable drift.

\begin{lemma}\label{xi0 lemma} Let $\xi_0$ be a finite $\cF_0$-measurable random variable and define
$$\varphi=\exp\Big(\xi_0M_T-\frac{\xi_0^2}{2}\<M\>_T\Big).
$$
Then $d\bP_{\xi_0}:=\varphi\,d\bP$ defined a probability measure $\bP_{\xi_0}$ under which $W^H$ is a fractional Brownian motion with constant drift $\xi_0$. That is, the law of $(W^H_t)_{0 \le t \le T}$ under $\bP_{\xi_0}$ is the same as that of $(W^H_t+\xi_0t)_{0\le t\le T}$ under $\bP$.
\end{lemma}

\begin{proof} Our assumption that $W^H$ is a fractional Brownian motion with respect to $(\cF_t)$ implies that the increments of $M$, and hence $M$ itself, are independent of $\cF_0$. Thus, $\xi_0$ and $M$ are independent, which gives $\bE[\varphi]=\bE[\bE[\varphi|\cF_0]]=1$. Thus, $\bP_{\xi_0}$ is indeed a probability measure. 

Next, it follows from \eqref{BM B def} and \eqref{W from B eq} that $W^H$ is independent of $\cF_0$. Moreover, it was shown in  \cite[Theorem 4.1]{Norrosetal1999} that the assertion holds if $\xi_0$ is a constant. Thus, by putting these two facts together and conditioning on $\cF_0$, one  sees that the assertion also holds for $\cF_0$-measurable $\xi_0$.
\end{proof}

The preceding lemma allows us to assume without loss of generality that $\xi_0=0$ $\bP$-a.s. To see why, we  rewrite \eqref{XH eq}  as
\begin{equation}\label{x0=0 eq}
	X_t = \left(W^H_t +\xi_0 t\right) + \int_{0}^{t}(\xi_s - \xi_0)\,ds.
\end{equation}
\Cref{xi0 lemma} implies that the law of $W^H$ is equivalent to the one of $(W^H_t + \xi_0 t)_{0\le t\le T}$.  Since absolute continuity is a transitive relation, our claim follows.

In the sequel, we shall work in a strictly pathwise manner for a fixed sample path $\xi$ that satisfies our assumptions. Under these assumptions, we may define  in analogy to \eqref{M and Y eq},
\begin{equation*}
\eta_t=\int_0^ts^{1/2-H}\xi_s\,ds\quad\text{and}\quad  \mu_t=\int_0^tw(t,s)\xi_s\,ds.
\end{equation*}
The function $\eta$ is clearly absolutely continuous with  progressively measurable derivative $\eta'_s=s^{1/2-H}\xi_s$, and so we have
\begin{equation}\label{Abel eq for eta}
\mu_t=c_1\int_0^t(t-s)^{1/2-H}s^{1/2-H}\xi_s\,ds=c_1\int_0^t(t-s)^{1/2-H}\eta'_s\,ds,
\end{equation}
which is the analogue of the first relation in \eqref{M and Y relations eq}. An analogue of the second relation in \eqref{M and Y relations eq} is proved in the following lemma. 

\begin{lemma}\label{eta lemma} Under the assumptions of \Cref{fBM KadotaShepp thm} and the additional assumption $\xi_0=0$, the function $t\mapsto \mu_t$ is absolutely continuous with bounded and progressively measurable derivative $\mu'_t$, and we have
\begin{equation}\label{eta analogy}
\eta_t=2H\int_0^t(t-s)^{H-1/2}\mu'_s\,ds.
\end{equation}
\end{lemma}

\begin{proof}We start with the case $H<1/2$, in which the derivative $\eta'$ is bounded. Equation \eqref{Abel eq for eta}, our assumption $H<1/2$, and the boundedness of $\eta_t'$ imply that  $t\mapsto \mu_t$ is also absolutely continuous with bounded and progressively measurable derivative
$$\mu_t'=c_1(1/2-H)\int_0^t(t-s)^{-H-1/2}\eta'_s\,ds.
$$
Thus, 
\begin{align*}
\int_0^t(t-s)^{H-1/2}\mu'_s\,ds&=c_1(1/2-H)\int_0^t(t-s)^{H-1/2}\int_0^s(s-r)^{-H-1/2}\eta'_r\,dr\,ds\\
&=c_1(1/2-H)\int_0^t\eta_r'\int_r^t(t-s)^{H-1/2}(s-r)^{-H-1/2}\,ds\,dr.
\end{align*}
After the change of variables $s=r+u(t-r)$, the inner integral becomes 
$$(t-r)^{1+H-1/2-H-1/2}\int_0^1(1-u)^{H-1/2}u^{-H-1/2}\,du=\Gamma(1/2+H)\Gamma(1/2-H).
$$
This gives \eqref{eta analogy}
in case $H<1/2$.

For $H>1/2$, recall that  we assume that  $\xi_t$ is H\"older continuous with some exponent $\alpha>2H-1$.  
 Let $ I_{0+}^{\beta}$ denote the fractional Riemann--Liouville integral operator of order $\beta>0$, i.e.,
$$I_{0+}^\beta f(t)=\frac1{\Gamma(\beta)}\int_0^t\frac{f(s)}{(t-s)^{1-\beta}}\,ds,\qquad t\in(0, T ],\ f\in L^1[0, T ].
$$
Furthermore, the fractional Riemann--Liouville derivative of order  $\beta\in(0,1)$ is given by
$$\cD^\beta_{0+}f(t)=\frac1{\Gamma(1-\beta)}\frac{d}{dt}\int_0^t\frac{f(s)}{(t-s)^\beta}\,ds,\qquad 0<t< T ,
$$
provided the expression on the right makes sense. 
When defining $\mu_t$ as in \eqref{Abel eq for eta} and letting  $\wt\xi_t:=t^{1/2-H}\xi_t$, we can write
\begin{equation}\label{mu frac int eq}
\mu_t=c_1\Gamma(3/2-H)I^{3/2-H}_{0+}\wt\xi(t).
\end{equation}
By \Cref{Hoelder lemma}, $\wt\xi_t$ is H\"older continuous with exponent $\lambda:=\alpha+1/2-H>H-1/2$. It follows from \cite[Lemma 13.2]{SamkoKilbasMarichev} that $\wt\xi\in  I_{0+}^{H-1/2}(L^1[0, T ])$. In addition, \eqref{mu frac int eq} and \cite[Theorem 2.3]{SamkoKilbasMarichev} yield that $\mu_t$ is absolutely continuous. Its derivative is given by 
\begin{equation}\label{muprime eq}
\mu'_t=c_1\frac{d}{dt}\int_0^t\frac{\wt\xi_s}{(t-s)^{H-1/2}}\,ds=c_1\Gamma(3/2-H)\cD_{0+}^{H-1/2}\wt \xi(t).
\end{equation}
To see that $\mu'$ is bounded, we use \cite[Lemma 13.1]{SamkoKilbasMarichev} to conclude that there exists a function $\psi$ that is H\"older continuous with exponent $\lambda-(H-1/2)>0$ such that 
$$\mu'_t=c_1\Gamma(3/2-H)
\cD_{0+}^{H-1/2}\wt\xi(t)=\frac{c_1\wt\xi_t}{t^{H-1/2}}+\psi(t)=c_1\xi_t+\psi(t).$$
Hence $\mu'_t$ is H\"older continuous and in particular bounded. Moreover, $\mu'_t$ is adapted by \eqref{muprime eq} and, hence, progressively measurable due to its continuity. 

By \cite[Theorem 2.4]{SamkoKilbasMarichev}, we can recover $\wt\xi$ from $\mu'$ as follows,
\begin{equation*}
c_1\Gamma(3/2-H)\wt\xi_t=I^{H-1/2}_{0+}\mu'(t)=\frac1{\Gamma(H-1/2)}\int_0^t\frac{\mu'_s}{(t-s)^{3/2-H}}\,ds.
\end{equation*}
Integrating with respect to $t$ and using Fubini's theorem yields,
\begin{align*}
\eta_t&=\int_0^t\wt\xi_s\,ds=\frac{1}{c_1\Gamma(3/2-H)\Gamma(H-1/2)}\int_0^t\int_0^r\frac{\mu'_s}{(r-s)^{3/2-H}}\,ds\,dr\\
&=\frac{1}{c_1\Gamma(3/2-H)\Gamma(H-1/2)}\cdot\frac1{H-1/2}\int_0^t(t-s)^{H-1/2}\mu'_s\,ds.
\end{align*}
This is \eqref{eta analogy} in our case $H>1/2$, because one checks that the constant in front of the integral is equal to $2H$.
\end{proof}

Let us now continue the proof of \Cref{fBM KadotaShepp thm}. \Cref{eta lemma}  allows us to make the following definition, which is analogous to  \eqref{BM B def},
\begin{equation}\label{beta def}
\beta_t:=\frac{2H}{c_H}\int_0^ts^{H-1/2}\mu'_s\,ds.
\end{equation}
By using the two identities \eqref{Abel eq for eta}
and \eqref{eta analogy} we can argue exactly as in the proof of 
  \cite[ Theorem 5.2]{Norrosetal1999} to get that 
\begin{equation}\label{beta to xi eq}
\int_0^t\xi_s\,ds=\int_0^t\zeta(t,s)\beta'_s\,ds.
\end{equation}

From  \eqref{M and Y relations eq} and \Cref{eta lemma} we now get the pathwise identity
$$Y^X_t:=\int_0^ts^{1/2-H}\,dX_s=Y_t+\eta_t.
$$
 Furthermore, we have
$$M^X_t:=\int_0^tw(t,s)\,dY^X_s=M_t+\mu_t,
$$
which shows that $M^X$ is a semimartingale with decomposition $M^X_t=M_t+\mu_t$.
Finally, we let
\begin{equation}\label{BX eq}
B^X_t:=\frac{2H}{c_H}\int_0^ts^{H-1/2}\,dM^X_s=B_t+\beta_t.
\end{equation}
By \eqref{W from B eq} and \eqref{beta to xi eq}, we have 
\begin{equation}\label{X from BX}
X_t=\int_0^t\zeta(t,s)\,dB^X_s,
\end{equation}
and so $X$ can be recovered in a pathwise manner from $B^X$. 

Next,  \Cref{eta lemma} implies that $\beta'_t=t^{H-1/2}\mu'_t$ is progressively measurable and satisfies 
\begin{equation}\label{beta square integrable eq}
\int_0^{ T }(\beta'_t)^2\,dt=\frac{4H^2}{c_H^2}\int_0^{ T }t^{2H-1}(\mu'_t)^2\,dt<\infty,\quad\text{$\bP$-a.s.,}
\end{equation}
 because $\mu'$ is bounded. Hence, \Cref{KadotaShepp thm} yields that the law of $B^X$ is absolutely continuous with respect to the law of $B$. Therefore, the assertion of \Cref{fBM KadotaShepp thm} follows from \eqref{W from B eq} and \eqref{X from BX}.
\qed\bigskip

The following elementary lemma was used in the proof.

\begin{lemma}\label{Hoelder lemma}If $f:[0, T ]\to\bR$ is H\"older continuous with exponent $\beta\in(0,1]$, satisfies $f(0)=0$, and $\alpha\in(0,\beta)$, then $\wt f(t):=t^{-\alpha}f(t)$  is H\"older continuous with exponent $\beta-\alpha$.
\end{lemma}

\begin{proof}Let $c\ge0$ be such that $|f(t)-f(s)|\le c|t-s|^{\beta}$ for all $s,t\in[0,1]$. Let $0\le s<t\le  T $ be given. If $s=0$, then $\wt f(s)=0$ and 
$\big|\wt f(t)-\wt f(s)\big|= t^{-\alpha}|f(t)|\le c t^{\beta-\alpha}$.
Now suppose that $s>0$. Using the inequalities $|f(t)-f(s)|\le c|t-s|^{\beta-\alpha}|t-s|^\alpha\le  c|t-s|^{\beta-\alpha}t^\alpha$ and  $|f(s)|=|f(s)-f(0)|\le cs^\beta$ yields 
\begin{align*}
\big|\wt f(t)-\wt f(s)\big|&\le t^{-\alpha}|f(t)-f(s)|+(t^{-\alpha}-s^{-\alpha})|f(s)|\le c|t-s|^{\beta-\alpha}+c|t^{-\alpha}-s^{-\alpha}|s^\beta.
\end{align*} 
For the rightmost term, we have 
\begin{align*}
|t^{-\alpha}-s^{-\alpha}|s^\beta=\alpha\int_s^ts^\beta r^{-\beta}r^{\beta-\alpha-1}\,dr\le\alpha\int_s^tr^{\beta-\alpha-1}\,dr=\frac{\alpha}{\beta-\alpha}(t^{\beta-\alpha}-s^{\beta-\alpha})\le \frac{\alpha}{\beta-\alpha}|t-s|^{\beta-\alpha}.\end{align*}
This completes the proof.\end{proof}

\subsection{Proof of \Cref{fBM Yershov theorem}}

We may once again assume without loss of generality that $x_0=0$; otherwise we replace $X_t$ with $X_t-x_0$ and $b$ with $b(\cdot+x_0)$. Moreover, as in \eqref{x0=0 eq}, we may assume without loss of generality that $b(X_0)=b(x_0)=b(0)=0$.

By  the one-to-one correspondence between $B^X$ and $X$, \eqref{BX eq} and \eqref{X from BX}, the equivalence of the laws of $X$ and $x_0+W^H$ will follow if we can show the equivalence of the laws of $B^X$ and $B$. Let $\xi_t:=b(X_t)$ and $\wt\xi_t:=b(W^H)$. Then $t\mapsto\xi_t$ and $t\mapsto\wt\xi_t$ are both $\bP$-a.s.~bounded. Moreover, for  $H>1/2$, we use the fact that the sample paths of $W^H$, and in turn those of $X$, are $\bP$-a.s.~H\"older continuous with exponent $\beta$ for every $\beta<H$; see, e.g., \cite[Section 1.15]{Mishura}. Therefore, the trajectories $t\mapsto b(X_t)=\xi_t$ and $t\mapsto b(W^H_t)=\wt\xi_t$ are H\"older continuous with exponent $\alpha\beta$. By assumption, $\alpha H>2H-1$, and hence there exists $\beta<H$ such that $\alpha\beta>2H-1$ as well. Thus, both $\xi$ and $\wt\xi$ satisfy the assumptions of  \Cref{fBM KadotaShepp thm}. Thus, we may apply \Cref{eta lemma}. It yields that both 
$$\mu_t=c_1\int_0^t(t-s)^{1/2-H}s^{1/2-H}\xi_s\,ds\quad\text{and}\quad \wt\mu_t=c_1\int_0^t(t-s)^{1/2-H}s^{1/2-H}\wt\xi_s\,ds
$$ 
are absolutely continuous with bounded and progressively measurable derivatives $\mu_t'$ and $\wt\mu_t'$. Thus, if we define $\beta$ and $\wt\beta$ via \eqref{beta def}, then $\int_0^T(\beta'_t)^2\,dt$ and $\int_0^T(\wt\beta'_t)^2\,dt$ are both $\bP$-a.s.~finite by 
\eqref{beta square integrable eq}.

We have seen in \eqref{BX eq} that $B^X_t=B_t+\beta_t$. Using in addition \eqref{beta def} and \eqref{X from BX}, we see that
\begin{equation*}
B^X_t=B_t+\int_0^t\gamma(s,B^X)\,ds
\end{equation*}
where 
\begin{equation*}
\gamma(t,B^X)=\frac{2H c_1}{c_H}t^{H-1/2}\frac{d}{dt}\int_0^t\frac{s^{1/2-H}}{(t-s)^{H-1/2}}b\bigg(\int_0^s\zeta(s,r)\,dB^X_r\bigg)\,ds=\beta'_t.
\end{equation*}
Hence, $B^X$ satisfies a stochastic integral equation  of the form \eqref{SDE}. Moreover, $\gamma(t,B)=\wt\beta'_t$. Hence, the conditions of \Cref{Yershov thm} are satisfied, and so the laws of $B$ and $B^X$ are equivalent.  
\qed

\begin{remark}\label{density remark 2} As mentioned in \Cref{density remark 1}, the proof of \Cref{fBM Yershov theorem} yields a formula for the density of the law of $(X_t)_{0\le t\le T}$ with respect to the one of $(W^H_t)_{0\le t\le T}$ in the spirit of \cite{Norrosetal1999}. It follows from the preceding proof, \eqref{W from B eq}, and  \eqref{X from BX} that this density must be the same as the one of the law of $B^X$ with respect to the law of $B$. In the proof of  \cite[Theorem 7.7]{LiptserShiryaevI}, it is shown that the latter is simply given by the standard Girsanov formula, i.e.,
$$\exp\bigg(\int_0^T\beta'_t\,dB_t-\frac12\int_0^T(\beta'_t)^2\,dt\bigg). 
$$
Using the formulas in the proof of \Cref{fBM KadotaShepp thm}, one can easily give several equivalent representations of this density.
\end{remark}

\bibliographystyle{plain}
	\bibliography{CTbook}
	\end{document}